\documentclass{article}
\usepackage{amsfonts}

\usepackage{multirow}
\usepackage{latexsym}
\usepackage{amsthm}
\usepackage{fancyheadings}
\usepackage{amsmath}

\newtheorem{Le}{Lemma}

\newtheorem{Co}{Corollary}
\newtheorem{The}{Theorem}
\newtheorem{Rem}{Remark}
\newtheorem{Pro}{Proposition}
\theoremstyle{definition}
\newtheorem{De}{Definition}

\newcommand{\s}{\mathfrak{S}}
\newcommand{\card}{{{\mathrm card}}}
\def\maj{{\scriptstyle \mathsf{MAJ}}}
\def\inv{{\scriptstyle \mathsf{INV}}}
\def\stat{{\scriptstyle \mathsf{STAT}}}
\def\st{{\scriptstyle \mathsf{ST}}}
\title{Lehmer code transforms and Mahonian statistics on permutations}

\author{
Vincent {\sc Vajnovszki}\\ 
{\small LE2I, Universit\'e de Bourgogne}\\
{\small BP 47870, 21078 Dijon Cedex, France}\\
{\small \tt vvajnov@u-bourgogne.fr}
}

\begin{document}

\maketitle

\begin{abstract}
In 2000 Babson and Steingr{\'\i}msson introduced
the notion of vincular patterns in permutations. 
They shown that essentially all well-known Mahonian
permutation statistics can be written as 
combinations of such patterns.
Also, they proved and conjectured that other 
combinations of vincular patterns are still Mahonian.
These conjectures were proved later:
by Foata and Zeilberger in 2001, and 
by Foata and Randrianarivony in 2006.

In this paper we give an alternative proof 
of some of these results.
Our approach is based on permutation codes which, like Lehmer's code,
map bijectively permutations onto subexcedant sequences.
More precisely, we give several code transforms (i.e., bijections
between subexcedant sequences) which when applied to
Lehmer's code yield new permutation codes which count occurrences 
of some vincular patterns.

\end{abstract}

\section{Introduction}

An alternative way to represent a permutation $\pi=\pi_1\pi_2\cdots\pi_n\in\s_n$ 
is by its permutation code $t_1t_2\cdots t_n$, which is a subexcedant sequence.
A classical example of permutation code is the Lehmer code, where
each $t_i$ is the number of entries in $\pi$ larger than $\pi_i$
and on its left.
The $\inv$ statistic on permutations is related to Lehmer code by $\inv \,\pi=\sum_{i=1}^nt_i$.

A code transform is a bijection from subexcedant sequences onto itself.
We give several code transforms and show that 
most of pattern-involvement based statistics introduced in \cite{BabSteim} 
are related to transforms of Lehmer code in the same way as 
$\inv$ is related to Lehmer code.
These results are summarized in the table at the end of this paper

\subsection{Permutation patterns}
A permutation $\sigma\in\s_k$ is a classical pattern of the permutation 
$\pi\in\s_n$, $k\leq n$, if there is a
sequence $1\leq i_1<i_2<\cdots<i_k\leq n$ such that
$\pi_{i_1}\pi_{i_2}\cdots\pi_{i_k}$ is order-isomorphic to
$\sigma$. 
Vincular patterns, introduced in \cite{BabSteim}, are generalizations
of classical patterns where:
\begin{itemize}
\item Two adjacent letters may or may not be separated
by a dash.  The absence of a dash between two adjacent letters 
means that the corresponding letters in the permutation
must be adjacent;  
\item Patterns may begin and/or end with square brackets. 
These indicate that they are
required to begin at the first letter in a permutation and/or end at
the last letter.  
\end{itemize}

In the following patterns will be written as words over the alphabet
$\{a,b,c,\ldots\}$ based on the usual ordering
$a<b<c\cdots$.

\subsection{Mahonian statistics and pattern involvement}
A {\em statistic} on $\s_n$ is an association of an integer
to each permutation in $\s_n$. Classical examples of statistics are
$\inv$ and  $\maj$ defined as

\begin{itemize}
\item[]
$
\inv\, \pi =\text{card}\, \{(i,j)\ |\ 1\leq i<j\leq n, \pi_i>\pi_j\},
$
\item[]
$
\displaystyle 
\maj\, \pi = \mathop{\sum_{1\leq i <n}}_{\pi_i>\pi_{i+1}} i.
$
\end{itemize}

A statistic $\st$ on $\s_n$ is {\it Mahonian} if it has the same distribution
as $\inv$, that~is 

$$
{\mathrm  \card}\,\{\pi\in\s_n \,|\, \st\, \pi = k\}=
{\mathrm  \card}\,\{\pi\in\s_n \,|\, \inv\,  \pi = k\},
$$
for any $k\geq 0$, and it is well known that $\maj$ is a Mahonian statistic.

For a permutation $\pi$ and a set of patterns $\{\sigma,\tau,\ldots \}$,
we denote be $(\sigma+\tau+\cdots)\,\pi$ the number of occurrences of  
these patterns in $\pi$, and $(\sigma+\tau+\cdots)$ becomes 
a permutation statistic. For example
$$
\inv\,\pi =(a-b)\,\pi,
$$    
and 
$$
\maj\,\pi=((a-cb)+(b-ca)+(c-ba)+(ba))\,\pi,
$$
and both statistics are Mahonian.

\medskip

In order to count the number of occurrences of the pattern $\sigma$ in
$\pi$ we introduce the notion of 
{\it pointed pattern}.
A pointed pattern 
is the pattern $\sigma$ together with a privileged element,
say the $\ell$th one; and we denote by 
$\sigma_1\sigma_2\cdots\underline{\sigma_\ell}\cdots\sigma_k$
such a pattern.
Often, when the privileged element is understood or does 
no matter we denote simply by $\underline{\sigma}$ a pointed pattern
if the underlining permutation is $\sigma$.
With these notations,
$(\sigma_1\sigma_2\cdots\underline{\sigma_\ell}\cdots\sigma_k)_i\,\pi$
denotes the number of occurrences of the pattern 
$\sigma$ in the permutation $\pi$, where the role of $\sigma_{\ell}$ 
is played by $\pi_i$. 
For example,  if $\pi=2\,4\,5\,1\,3\,6$, then $(b-\underline{a}c)_5\,\pi=2$, and
$(b-a\underline{c})_5\,\pi=1$.
Clearly, 
 
$$
(\sigma)\,\pi=
\sum_{i=1}^n(\underline{\sigma})_i\,\pi
$$
for any pointed pattern $\underline{\sigma}$ corresponding to 
$\sigma$.

\subsection{Permutation codes}

An integer sequence $t_1t_2\cdots t_n$ is said to be 
{\em subexcedant} if $0\leq t_i\leq i-1$ for $1\leq i\leq n$, and the set of
all length-$n$ subexcedant sequences is denoted by $S_n$; so 
$S_n=
\{0\}\times \{0,1\}\times\cdots\times\{0,1,\ldots,n-1\}.
$
Clearly, $\s_n$ is in bijection with $S_n$, and any such
bijection is called {\it permutation code}.

The {\em Lehmer code} \cite{Lehmer} is a classical
example of permutation code and it will be our 
starting point for the construction of other several
permutation codes.
This code bijectively maps each permutation 
onto a subexcedant sequence of same length.

\begin{De}
For $\pi=\pi_1\pi_2\cdots\pi_n\in\s_n$, the Lehmar code $L(\pi)$
of $\pi$ is the subexcedant sequence $t_1t_2\cdots t_n$
where, for all $i$, $1\leq i\leq n$, $t_i$ is the number of inversions $(j,i)$ in $\pi$,
that is, number of $j$ with $\pi_j>\pi_i$ but $j<i$.
\end{De}

For example $L(5\,2\,1\,6\,4\,3)=0\,1\,2\,0\,2\,3$.
Some permutation codes can be obtained from 
pattern occurrences, and this is the case for the Lehmer code
$t_1t_2\cdots t_n$ of a permutation~$\pi$:
\begin{equation}
\label{lehmer1}
t_i= (b-\underline{a})_i\,\pi
\end{equation}
and so $\inv\,\pi=(b-a)\,\pi$.
Alternatively, 
\begin{equation}
\label{lehmer2}
t_i=
((b-c\underline{a})+(c-a\underline{b})+(c-b\underline{a})+(b\underline{a}) )_i\,\pi,
\end{equation}
and $\inv\,\pi=((b-ca)+(c-ab)+(c-ba)+(ba) )\,\pi$.

\medskip

Let $\{\underline{\sigma},\underline{\tau},\ldots\}$
be a set of pointed patterns. If the function 
$$
\pi\mapsto t_1t_2\cdots t_n
$$ 
where
$$
t_i=(\underline{\sigma}+\underline{\tau}\cdots)_i\,\pi,\ {\rm for}\ 1\leq i\leq n,
$$
is a permutation code, then we say that the set of patterns
$\{\sigma,\tau,\ldots\}$ induces a permutation code.
For example, relations (\ref{lehmer1}) and (\ref{lehmer2})
show that the Lehmer code $L$ is induced both by 
the set of patterns $\{b-a\}$
and $\{c-ba,b-ca,c-ab,ba\}$.

For each permutation code $\pi\mapsto t_1t_2\cdots t_n$
we can associate naturally a Mahonian statistic $\st$ on $\s_n$, defined by 
$$
\st\, \pi = \sum_{i=1}^n t_i.
$$
In addition, if the set of patterns $\{\sigma,\tau,\ldots\}$ induces a 
permutation code, then the statistic 
$$
\pi\mapsto (\sigma+\tau\cdots)\,\pi
$$ 
is Mahonian.
%
%
We will see that, like $\{b-a\}$ and $\{c-ba,b-ca,c-ab,ba\}$,
several other patterns induce permutation codes,
and so these patterns are Mahonian, see Table \ref{table}.

%
%
%

%

\section{Code transforms}

We call a bijection from  $S_n$ onto itself a {\it code transform}.
Below we give six functions and we show that they are code transforms,
and thus each of them applied to the Lehmer code yields still
a permutation code. 

\begin{De}
The functions  
$\Delta,\Gamma,\Theta,\Lambda,\Theta,\Upsilon,\Psi: S_n\rightarrow S_n$
are defined as follows. If $t=t_1t_2\cdots t_n\in S_n$, then 
\begin{itemize}
\item $\Delta(t)=s_1s_2\cdots s_n$ is defined by 
 $
  s_i=\left\{ \begin {array}{ccc}
  (t_{i}-t_{i+1}) \mod i    & {\rm if} &  1\leq i<n \\
  t_n         & {\rm if} & i=n. 
  \end {array}
  \right.
  $ 
\item $\Gamma(t)=s_1s_2\cdots s_n$ is defined by 
  $
  s_i=\left\{ \begin {array}{ccc}
   0         & {\rm if} & i=1\\
  (t_{i-1}-t_{i}) \mod i   & {\rm if} &  1<i\leq n.
  \end {array}
  \right.
  $
\item $\Theta(t)=s_1s_2\cdots s_n$ is defined by
  $
  s_i=\left\{ \begin {array}{ccc}
  0                & {\rm if} & i=1 \\
  t_{i-1}-t_{i}  & {\rm if} & t_{i-1}\geq t_i {\rm\ and\ } 1<i\leq n \\
  t_{i}            & {\rm if} & t_{i-1}<    t_i {\rm\ and\ } 1<i\leq n.
  \end {array}
  \right.
  $
\item $\Lambda(t)=s_1s_2\cdots s_n$ is defined by 
  $
  s_i=\left\{ \begin {array}{ccc}
  0                & {\rm if} & i=1 \\
  t_i  & {\rm if} & t_{i-1}\geq t_i {\rm\ and\ }  1<i\leq n \\
   i+t_{i-1}-t_{i}           & {\rm if} & t_{i-1}<    t_i {\rm\ and\ } 1<i\leq n.
  \end {array}
  \right.
  $
\item 
  $\Upsilon(t)=s_1s_2\cdots s_n$ is defined by
$
  s_i=\left\{ \begin {array}{ccc}
  i-t_i-1     & {\rm if} & t_i     <    t_{i+1} {\rm\ and\ } 1\leq i<n \\
  t_i-t_{i+1} & {\rm if} & t_i \geq t_{i+1}     {\rm\ and\ } 1\leq i<n\\
  t_n         & {\rm if} & i=n. 
\end {array}
\right.
$  
\item 
  $\Psi(t)=s_1s_2\cdots s_n$ is defined by
 $
s_i=\left\{ \begin {array}{ccc}
  t_{i+1}-t_i-1     & {\rm if} & t_i     <    t_{i+1} {\rm\ and\ } 1\leq i<n \\
  t_i & {\rm if} & t_i \geq t_{i+1}     {\rm\ and\ } 1\leq i<n  \\
  t_n         & {\rm if} & i=n. 
\end {array}
\right.
$  
\label{six_fun}
\end{itemize}

\end{De}

\noindent
See Figure \ref{fig_double}.a for several examples.

\begin{Pro}
The six functions given in Definition \ref{six_fun} are bijections.

\end{Pro}
\begin{proof}
It is enough to prove that each of these functions is injective,
cardinality reasons complete the proof. 
The injectivity of $\Delta$ and $\Gamma$ is routine.

In the next we consider two different sequences 
$t=t_1t_2\cdots t_n$ and $t'=t'_1t'_2\cdots t'_n$ in $S_n$.

\medskip
\noindent
{\it The injectivity of $\Theta$ and  $\Lambda$.}
Let $i$ be the leftmost position
where $t$ and $t'$ differ. 
The sequences $s_1s_2\cdots s_n=\Theta(t)$
and $s'_1s'_2\cdots s'_n=\Theta(t')$ differ also in position $i$.
Indeed, if $t_i\leq t_{i-1}=t'_{i-1}< t'_i$, then 
$s_i\leq t_{i-1}<s'_i$ and
so $s_i\neq s'_i$. The other cases are equivalent or trivial, and
the injectivity of $\Lambda$ is similar.

%
%

\medskip
\noindent
{\it The injectivity of $\Upsilon$.}
Let $i$ be the rightmost position
where $t$ and $t'$ differ. 
The sequences $s_1s_2\cdots s_n=\Upsilon(t)$
and $s'_1s'_2\cdots s'_n=\Upsilon(t')$ differ also in position $i$.
Indeed, if $i<n$ and $t_i\geq t_{i+1}=t'_{i+1}> t'_i$, then 
$s_i\leq i-1-t_{i+1}<s'_i$ and
so $s_i\neq s'_i$. The other cases are equivalent or trivial.

\medskip
\noindent
{\it The injectivity of $\Psi$.}
Let $i$ be the rightmost position where $t$ and $t'$ differ. 
The sequences $s_1s_2\cdots s_n=\Psi(t)$
and $s'_1s'_2\cdots s'_n=\Psi(t')$ differ also in position $i$.
Indeed, if $i<n$ and $t_i\geq t_{i+1}=t'_{i+1}> t'_i$, then $s_i\geq t_{i+1}>s'_i$ and
so $s_i\neq s'_i$. The other cases are equivalent or trivial.

\end{proof}

\noindent
It is easy to check the following.

\begin{Rem}
The inverse of $\Delta$, $\Gamma$, $\Theta$, $\Lambda$, $\Upsilon$ and 
$\Psi$ are given below. If $s=s_1s_2\cdots s_n\in S_n$, then

\begin{itemize}

\item $\Delta^{-1}(s)=t_1t_2\cdots t_n$ with 
$
t_i=\left\{ \begin {array}{ccc}
  s_n & {\rm if} &  i = n \\
  (t_{i+1}+s_i) \mod i         & {\rm if} &  1\leq i<n,
\end {array}
\right.
$ 
\item $\Gamma^{-1}(s)=t_1t_2\cdots t_n$ with 
$
t_i=\left\{ \begin {array}{ccc}
  s_1 & {\rm if} &  i = 1 \\
  (t_{i-1}-s_i) \mod i         & {\rm if} &  1< i\leq n. 
\end {array}
\right.
$

\item $\Theta^{-1}(s)=t_1t_2\cdots t_n$ with
  $
  t_i=\left\{ \begin {array}{ccc}
  0                & {\rm if} & i=1 \\
  s_{i}  & {\rm if} & t_{i-1}< s_i {\rm\ and\ } 1< i\leq n \\
  t_{i-1}-s_{i}            & {\rm if} & t_{i-1}\geq    s_i {\rm\ and\ } 1< i\leq n.
  \end {array}
  \right.
  $
  
\item $\Lambda^{-1}(s)=t_1t_2\cdots t_n$ with
  $
  t_i=\left\{ \begin {array}{ccc}
  0                & {\rm if} & i=1 \\
  i+t_{i-1}-s_{i}  & {\rm if} &  t_{i-1}<s_i  {\rm\ and\ } 1< i\leq n \\
  s_{i}            & {\rm if} &  t_{i-1}\geq s_i \ {\rm\ and\ }1< i\leq n.
  \end {array}
  \right.
  $
   
   \item $\Upsilon^{-1}(s)=t_1t_2\cdots t_n$ with
  $
  t_i=\left\{ \begin {array}{ccc}
  s_i                & {\rm if} & i=n \\
  t_{i+1}+s_i  & {\rm if} & t_{i+1}+s_i\leq i-1 {\rm\ and\ } 1\leq i< n \\
  i-1-s_{i}    & {\rm if} & t_{i+1}+s_i> i-1 {\rm\ and\ }  1\leq i< n.
  \end {array}
  \right.
  $
     
   \item $\Psi^{-1}(s)=t_1t_2\cdots t_n$ with
  $
  t_i=\left\{ \begin {array}{ccc}
  s_i                & {\rm if} & i=n \\
  s_i  & {\rm if} & s_i\geq t_{i+1} {\rm\ and\ } 1\leq i< n \\
  t_{i+1}-s_i-1 & {\rm if} & s_i< t_{i+1} {\rm\ and\ } 1\leq i< n.
  \end {array}
  \right.
  $
   
\end{itemize}

\end{Rem}

 
\section{Main results}

In this section we will prove the Mahonity of some patterns.
The results are stated in Theorems \ref{S_2}-\ref{third} and summarized
in Table \ref{table}; all of them are already known
\cite{BabSteim,Foata_Zeilberger,Foata_Randrianarivony}. 
The novelty consists in the unified approach
based on permutation codes, described briefly as:
\begin{itemize}
\item Find a convenient set of pointed patterns 
$\{\underline{\sigma},\underline{\tau},\ldots\}$
corresponding to the set of patterns $\{\sigma,\tau,\ldots\}$,
\item Show that for all $\pi\in\s_n$ the 
map $\pi\mapsto t_1t_2\cdots t_n$ 
with $t_i=(\underline{\sigma}+\underline{\tau}+\cdots)_i\,\pi$,
$1\leq i\leq n$, is a permutation code
based on (a transform of) the Lehmer code of $\pi$.
\end{itemize}

\noindent
This technique was initiated by the author in \cite{Vaj_11}
where the transform $\Delta$ is introduced. 
Before giving our first theorem, we need some further considerations
on this transform.
For a permutation $\pi\in\s_n$ with $L(\pi)$ its Lehmer code,
we call $\Delta(L(\pi))\in S_n$,
the {\it McMahon code} of $\pi$.
This is justified by the following result which is a consequence
of Theorem 13 and Corollary 6 in \cite{Vaj_11}.
\begin{Pro}
\label{Pro_Vaj_11}
If $s_1s_2\cdots s_n$ is the  McMahon code of 
$\pi\in\s_n$, then $\maj\,\pi=\sum_{i=1}^n s_i$. 
\end{Pro}

\noindent
Also in \cite{Vaj_11} is given the next corollary, 
expressed here in terms of patterns involvement.

\begin{Co}
\label{Cons_Mah}
For any $\pi\in\s_n$, the McMahon code $s=s_1s_2\cdots s_n$ of $\pi$ is 
given by:

$$
s_i=\left\{ \begin {array}{ccc}
((a-\underline{c}b)+(b-\underline{a}c)+(c-\underline{b}a))_i\,\pi   & {\rm if} &  i \neq n \\
(b-a]\,\pi            & {\rm if} &  i=n. 
\end {array}
\right.
$$ 
\end{Co}

\noindent
A consequence of Corollary \ref{Cons_Mah} is

\begin{Co}
\label{maj]} $ $
For any $\pi\in\s_n$ we have

\begin{enumerate}
\item 
$\maj\,\pi=((a-cb)+(b-ac)+(c-ba)+(b-a])\,\pi$,
\item 
If the McMahon code of $\pi$ is $s_1s_2\cdots s_n$, then
$$
((a-cb)+(b-ac)+(c-ba))\,\pi=\sum_{i=1}^{n-1}s_i.
$$
\end{enumerate}
\end{Co}

In \cite[Theorem 12]{Vaj_11} is given 
an algorithmic meaning of the McMahon code $s_1s_2\cdots s_n$
of $\pi\in\s_n$:
$\pi$ can be obtained from the identity $\iota=1\,2\,\ldots \,n\in\s_n$
by iteratively performing on $\iota$, $s_i$
right circular shifts of its length-$i$ prefix,
for $i=n,n-1,\ldots,2,1$.
For example, the construction of $\pi=5\,2\,1\,6\,4\,3\in\s_6$ with its 
McMahon code $0\,1\,2\,2\,4\,3$ is given in Figure \ref{fig_double}.b. 

\medskip

\begin{figure}
\begin{center}
\begin{tabular}{cc}
\begin{tabular}{|rcl|}
\hline
$\pi$              & $=$ & $5\,2\,1\,6\,4\,3$ \\
$L(\pi)$           & $=$ & $0\,1\,2\,0\,2\,3$ \\
$\Delta(L(\pi))$   & $=$ & $0\,1\,2\,2\,4\,3$ \\
$\Gamma(L(\pi))$   & $=$ & $0\,1\,2\,2\,3\,5$ \\
$\Theta(L(\pi))$   & $=$ & $0\,1\,2\,2\,2\,3$ \\
$\Lambda(L(\pi))$  & $=$ & $0\,1\,2\,0\,3\,5$ \\
$\Upsilon(L(\pi))$ & $=$ & $0\,0\,2\,3\,2\,3$ \\
$\Psi(L(\pi))$     & $=$ & $0\,0\,2\,1\,0\,3$ \\
\hline
\end{tabular}
&
\begin{tabular}{|lc|} 
\hline
after $s_i$ right circular shifts & \\
of the length-$i$ prefix & $(i,s_i)$ 
  \\ \hline
$1\,2\,3\,4\,5\,6$ &  \\
$4\,5\,6\,1\,2\,3$ & $(6,3)$ \\
$5\,6\,1\,2\,4\,3$ & $(5,4)$ \\
$1\,2\,5\,6\,4\,3$ & $(4,2)$ \\
$2\,5\,1\,6\,4\,3$ & $(3,2)$ \\
$5\,2\,1\,6\,4\,3$ & $(2,1)$ \\
$5\,2\,1\,6\,4\,3$ & $(1,0)$ \\
\hline
\end{tabular}  
 
\\
\\
(a) & (b)
\end{tabular}
\caption{\label{fig_double}
(a) The permutation $\pi$ together with its
Lehmer code and its transforms.
(b) The construction of $\pi=5\,2\,1\,6\,4\,3$
from its McMahon code $s=\Delta(L(\pi))=0\,1\,2\,2\,4\,3$. 
 } 
\end{center}
\end{figure}

With this algorithmic interpretation of McMahon code we have the 
following remark.

\begin{Rem}
\label{last_pos}
If $\sigma,\tau\in\s_n$ are two permutations with their  McMahon codes 
differing only in the last position, then $\sigma_i\neq\tau_i$
for all $i$, $1\leq i\leq n$.
\end{Rem}

\begin{The} 
\label{S_2}
The following statistic 
(statistic $S_2$ in \cite[Conjecture 11]{BabSteim})
is Mahonian
$$
(a-cb)+(b-ac)+(c-ba)+[b-a).
$$
\end{The}
\begin{proof}
To a permutation $\pi=\pi_1\pi_2\cdots \pi_n\in\s_n$ 
with its McMahon code $p_1p_2\cdots p_{n-1}p_n$ 
we associate a subexcedant sequence $p_1p_2\cdots p_{n-1}(\pi_1-1)$.
Clearly, by the second point of Corollary~\ref{maj]}
$$
((a-cb)+(b-ac)+(c-ba)+[b-a))\,\pi=\sum_{i=1}^{n-1}p_i+(\pi_1-1).
$$
Now we show that the map 
\begin{equation}
\label{a_map}
\pi\mapsto p_1p_2\cdots p_{n-1}(\pi_1-1)
\end{equation}
is an injection and so (by cardinality reasons)
a permutation code.

\noindent
Let $\sigma\neq\tau$ be two permutations in $\s_n$ and let $i$ be the 
leftmost position where $s_1s_2\cdots s_n$ and
$t_1t_2\cdots t_n$, their McMahon codes, differ. 
If $i<n$, then the subexcedant sequences corresponding to
$\sigma$ and $\tau$ (defined in the map in relation (\ref{a_map}))
differ also in position $i$.
If $i=n$, then by Remark~\ref{last_pos}, $\sigma_1\neq\tau_1$
and again, the subexcedant sequences corresponding to
$\sigma$ and $\tau$ are different.
\end{proof}

The reduction of a sequence of $n$ distinct integers
is the permutation in $\s_n$ 
obtained by replacing the smallest member by 1, the
second-smallest by 2, \ldots , and the largest by $n$. 

\begin{Rem} 
\label{rem_maj}
Let $\pi=\pi_1\pi_2\cdots \pi_n\in\s_n$.
If $\tau\in\s_{n-1}$ is the reduction of 
$\pi_2\cdots \pi_n$, then 
\begin{eqnarray*}
((a-cb)+(b-ca)+(c-ba))\,\pi & = & \maj\,\tau.
\end{eqnarray*}

\end{Rem}
\begin{proof}
For all $i\geq 2$, 
$((a-\underline{c}b)+(b-\underline{c}a)+(c-\underline{b}a))_i\,\pi$
equals $i-1$ if $i$ is a descent in $\pi$ 
(and so, if $i-1$ is a descent in $\tau$) and $0$ otherwise.
Now, summing for all $i$, $2\leq i\leq n$, the desired relation holds.
\end{proof}

\begin{The} 
\label{TS_4}
The following statistic 
(statistic $S_4$ in \cite[Conjecture 11]{BabSteim}) 
is Mahonian
$$
(a-cb)+(b-ca)+(c-ba)+[b-a).
$$
\end{The}
\begin{proof}
For $\pi=\pi_1\pi_2\cdots \pi_n\in\s_n$ let $\tau\in\s_{n-1}$
be the reduction of $\pi_2\cdots \pi_n$ and  
define $s=s_1s_2\cdots s_n\in S_n$ by:
\begin{itemize}
\item $s_1s_2\cdots s_{n-1}$ is the McMahon code of $\tau$, and 
\item $s_n=\pi_1-1$.
\end{itemize}
First, the map $\pi\mapsto s$ is a permutation code. 
Indeed, $s_1s_2\cdots s_{n-1}s_n\in S_n$ and
the prefix $s_1s_2\cdots s_{n-1}$ uniquely determines $\tau$, 
which together with  $s_n$ determines $\pi$.
Now we show that 
\begin{equation}
\label{S_4}
\sum_{i=1}^n s_i=((a-cb)+(b-ca)+(c-ba)+[b-a))\,\pi.
\end{equation}

\noindent
Using Remark \ref{rem_maj} we have
\begin{eqnarray*}
((a-cb)+(b-ca)+(c-ba))\,\pi & = & \maj\,\tau\\
                            & = & \sum_{i=1}^{n-1} s_i,
\end{eqnarray*}
and since $s_n=\pi_1-1=[b-a)\,\pi$ relation (\ref{S_4}) holds.
\end{proof}

\begin{The} 
\label{3stats}
The following statistics (defined in \cite[Proposition 9]{BabSteim}
or equivalent to them) are Mahonian. 
\begin{enumerate}
\item $\stat  =(a-cb)+(b-ac)+(c-ba)+(ba),$ 
\item $\stat'=(b-ac)+(b-ca)+(c-ba)+(ba),$
\item $\stat''=(a-cb)+(c-ab)+(c-ba)+(ba).$
\end{enumerate}
\end{The}
\begin{proof}
Let $\pi\in\s_n$ and $t$ its Lehmer code. We will use the remark that
$\pi_{i-1}<\pi_i$ if and only if $t_{i-1}\geq t_i$.

\noindent 1.
Let $s_1s_2\cdots s_n=\Gamma(t)$.
It is routine to check that

\begin{equation}\label{b_au(c)}
(b-a\underline{c})_i\,\pi
=\left\{ \begin {array}{ccc}
   t_{i-1}-t_i & {\rm if} &  t_{i-1}\geq t_i\\
   0           & {\rm if} &  t_{i-1}< t_i,
\end {array}
\right.
\end{equation}
and
\begin{equation}\label{a_cb+}
((a-c\underline{b})+(c-b\underline{a})+(b\underline{a}))_i\,\pi
=\left\{ \begin {array}{ccc}
    0 & {\rm if} &  t_{i-1}\geq t_i\\
   i+t_{i-1}-t_i & {\rm if} &  t_{i-1}< t_i,
\end {array}
\right.
\end{equation}
for $1< i\leq n$.

\noindent
Summing both relations and using the definition of $\Gamma$ we have 
$$
((a-c\underline{b})+(b-a\underline{c})+(c-b\underline{a})+(b\underline{a}))_i\,\pi=s_i,
$$
and thus
$$
((a-cb)+(b-ac)+(c-ba)+(ba))\,\pi=\sum_{i=1}^ns_i.
$$

\medskip

\noindent 2.
Let $s_1s_2\cdots s_n=\Theta(t)$. Similarly, we have 

$$((b-c\underline{a})+(c-b\underline{a})+(b\underline{a}))_i\,\pi
=\left\{ \begin {array}{ccc}
    0   & {\rm if} &  t_{i-1}\geq t_i\\
    t_i & {\rm if} &  t_{i-1}< t_i,
\end {array}
\right.
$$
for $1< i\leq n$.

\noindent
Summing this relation with (\ref{b_au(c)}) and using the definition of $\Theta$ we have
$$
((b-a\underline{c})+(b-c\underline{a})+(c-b\underline{a})+(b\underline{a}))_i\,\pi=s_i,
$$
and thus
$$
((b-ac)+(b-ca)+(c-ba)+(ba))\,\pi=\sum_{i=1}^ns_i.
$$

\medskip

\noindent 3.
Now let $s_1s_2\cdots s_n=\Lambda(t)$.

$$
(c-a\underline{b})_i\,\pi
=\left\{ \begin {array}{ccc}
   t_i & {\rm if} &  t_{i-1}\geq t_i\\
   0   & {\rm if} &  t_{i-1}< t_i,
\end {array}
\right.
$$
for $1< i\leq n$, which together with (\ref{a_cb+}) and the definition of $\Lambda$ gives

$$
((a-c\underline{b})+(c-a\underline{b})+(c-b\underline{a})+(b\underline{a}))_i\,\pi=s_i,
$$
and so
$$
((a-cb)+(c-ab)+(c-ba)+(ba))\,\pi=\sum_{i=1}^ns_i.
$$

\noindent
Since $\Gamma$, $\Theta$ and $\Lambda$ are code transforms, 
it results that the three statistics are Mahonian.
\end{proof}

Before proving our next theorem we need the following result.

\begin{Le} For any integer $n$ and permutation $\pi\in\s_n$ we have
\label{lemma_dvr}
$$((b-ca)+(ba))\,\pi=((b-ac)+(b-a])\,\pi.$$
\end{Le}
\begin{proof}$ $

\begin{tabular}{rclcr}
$\maj\,\pi$ & $=$ & $((a-cb)+(b-ca)+(c-ba)+(ba))\,\pi$      &\hspace{1cm} & by definition\\
            & $=$ & $((a-cb)+(b-ac)+(c-ba)+(b-a])\,\pi$ &\hspace{1cm} & by
	    Corollary \ref{maj]}.1
\end{tabular}
%
%
%
%
%

\noindent
and the result holds.
\end{proof}

\begin{The}
\label{second}
The following statistic
(second statistic 
in \cite[Conjecture 8]{BabSteim}) is Mahonian
$$
(a-cb)+(b-ca)+(b-ca)+(ba).
$$
\end{The}

\begin{proof}

Let $\pi\in\s_n$ with its Lehmer code $t=t_1t_2\cdots t_n$ and 
$s=s_1s_2\cdots s_n=\Upsilon(t)$.
As previously, it is easy to check that

$$
((a-\underline{c}b)+(b-\underline{c}a)_i\,\pi=
\left\{ \begin {array}{ccc}
    i-t_i-1 & {\rm if} &  t_i< t_{i+1}\\
    0       & {\rm if} &  t_i\geq t_{i+1},
\end {array}
\right.
$$
and
$$
(b-\underline{a}c)_i\,\pi=
\left\{ \begin {array}{ccc}
    0 & {\rm if} &  t_i< t_{i+1}\\
    t_i-t_{i+1}       & {\rm if} &  t_i\geq t_{i+1},
\end {array}
\right.
$$
for $1\leq i<n$.

\noindent
Summing both relations and using the definition of 
$\Upsilon$ we have 

$$
s_i=\left\{ \begin {array}{ccc}
((a-\underline{c}b)+(b-\underline{a}c)+(b-\underline{c}a))_i\,\pi   & {\rm if} &  i \neq n \\
(b-a]\,\pi            & {\rm if} &  i=n. 
\end {array}
\right.
$$

\noindent
Now, by Lemma \ref{lemma_dvr} and the previous relation we have  
\begin{eqnarray*}
((a-cb)+(b-ca)+(b-ca)+(ba))\,\pi & = & ((a-cb)+(b-ac)+(b-ca)+(b-a])\,\pi \\
 & = & \sum_{i=1}^n s_i,
\end{eqnarray*}
and since $\Upsilon$ is a code transform the result holds.
\end{proof}

\begin{The}
\label{third}
The following statistic (equivalent with the third one in \cite[Conjecture 8]{BabSteim}) 
is Mahonian
$$
(b-ca)+(b-ca)+(c-ab)+(ba).
$$
\end{The}
\begin{proof}
Let $\pi\in\s_n$ with its Lehmer code $t=t_1t_2\cdots t_n$ and 
$s=s_1s_2\cdots s_n=\Psi(t)$. We have

$$
((b-\underline{a}c)+(c-\underline{a}b))_i\,\pi=
\left\{ \begin {array}{ccc}
     0 & {\rm if} &  t_i< t_{i+1}\\
    t_i        & {\rm if} &  t_i\geq t_{i+1},
\end {array}
\right.
$$
and
$$
(b-\underline{c}a)_i\,\pi=
\left\{ \begin {array}{ccc}
    t_{i+1}-t_i-1 & {\rm if} &  t_i< t_{i+1}\\
    0         & {\rm if} &  t_i\geq t_{i+1},
\end {array}
\right.
$$
for $1\leq i<n$, which together with
the definition of $\Psi$ gives

$$
s_i=\left\{ \begin {array}{ccc}
((b-\underline{a}c)+(b-\underline{c}a)+(c-\underline{a}b))_i\,\pi   & {\rm if} &  i \neq n \\
(b-a]\,\pi            & {\rm if} &  i=n. 
\end {array}
\right.
$$ 

\noindent
Now, by Lemma \ref{lemma_dvr} and the previous relation we have

\begin{eqnarray*}
((b-ca)+(b-ca)+(c-ab)+(ba))\,\pi & = & ((b-ac)+(b-ca)+(c-ab)+(b-a])\,\pi \\
& = & \sum_{i=1}^n s_i,
\end{eqnarray*}
and since $\Psi$ is a code transform the result holds.
\end{proof}

\begin{table}
\small
\begin{center}
\begin{tabular}{|c|c|c|}
\hline
 \multicolumn{2}{|c|}{}         & transform   \\
 \multicolumn{2}{|c|}{statistic}                  & of Lehmer code  \\
 \multicolumn{2}{|c|}{}                  & (if any)  \\
\hline\hline
$\inv$ 
 & \begin{tabular}{l} 
    $(b-a)$ \\ 
    $(b-ca)+(c-ab)+(c-ba)+(ba)$
   \end{tabular}

 & 
          \\ \hline
$\maj$ &
$(a-cb)+(b-ca)+(c-ba)+(ba)$
& 
\\ \hline

   \begin{tabular}{c} 
    $\maj$\\ 
   (Proposition \ref{Pro_Vaj_11})
   \end{tabular}
   
   &
$(a-cb)+(b-ac)+(c-ba)+(b-a]$ 
& 
 
  $\Delta$
              
  \\ \hline

 \begin{tabular}{c} 
   Theorem \ref{S_2}  \\ 
   (statistic $S_2$ in \\
   \cite[Conjecture 11]{BabSteim})
   \end{tabular}
& $(a-cb)+(b-ac)+(c-ba)+[b-a)$ & 
\\ \hline
 \begin{tabular}{c} 
   Theorem \ref{TS_4}  \\ 
   (statistic $S_4$ in \\
   \cite[Conjecture 11]{BabSteim})
   \end{tabular}
&  $(a-cb)+(b-ca)+(c-ba)+[b-a)$ & 
\\ \hline


  \begin{tabular}{c} 
   Theorem \ref{3stats}.1 \\ 
   ($\stat$ in \\
   \cite[Proposition 9]{BabSteim})
   \end{tabular}
 & 
   $(a-cb)+(b-ac)+(c-ba)+(ba)$                 & 
   
   $\Gamma$
              
   \\ \hline

\begin{tabular}{c} 
   Theorem \ref{3stats}.2 \\ 
   ($\stat'$ in \\
   \cite[Proposition 9]{BabSteim})
   \end{tabular}  &
$(b-ac)+(b-ca)+(c-ba)+(ba)$                 & 
  $\Theta$
            \\ \hline

 \begin{tabular}{c} 
   Theorem \ref{3stats}.3 \\ 
   ($\stat''$ in \\
   \cite[Proposition 9]{BabSteim})
   \end{tabular}   &
$(a-cb)+(c-ab)+(c-ba)+(ba)$                 & 
  $\Lambda$
        \\ \hline
\begin{tabular}{c} 
   Theorem \ref{second} \\ 
   (second statistic in\\
   \cite[Conjecture 8]{BabSteim})
   \end{tabular}  &
$(a-cb)+(b-ca)+(b-ca)+(ba)$                 & 
$\Upsilon$
            \\ \hline
 \begin{tabular}{c} 
   Theorem \ref{third} \\ 
   (third statistic in\\
   \cite[Conjecture 8]{BabSteim})
   \end{tabular} 
 &
$(b-ca)+(b-ca)+(c-ab)+(ba)$                 &  
$\Psi$
              \\ \hline

\end{tabular}
\end{center}
\caption{\label{table} Pattern statistics together with their 
Lehmer code transforms.}

\end{table}

\end{document}